\documentclass[12pt]{article}
\usepackage{amsfonts}
\usepackage{amsmath}
\usepackage{amssymb}
\usepackage{amsthm}

\title{Set systems without a $3$-simplex}

\author{Michael E. Picollelli\footnote{
\small Department of Electrical \& Computer Engineering, University of Delaware, Newark, DE, USA.  E-mail: \texttt{mpicolle@udel.edu}}}
\date{}

\newcommand{\bin}[2]{\binom{#1}{#2}}
\newcommand{\lp}{\left(}
\newcommand{\rp}{\right)}
\newtheorem{theorem}{Theorem}

\newtheorem{lemma}{Lemma}

\newtheorem{claim}{Claim}

\newtheorem{conjecture}{Conjecture}
\newcommand{\ca}[1]{{\cal #1}}

\begin{document}

\maketitle

\begin{abstract}
A $3$-simplex is a collection of four sets $A_1,\ldots,A_4$ with empty intersection such that any three of them have nonempty intersection.  We show that the maximum size of a set system on $n$ elements without a $3$-simplex is $2^{n-1} + \binom{n-1}{0} + \binom{n-1}{1} + \binom{n-1}{2}$ for all $n \ge 1$, with equality only achieved by the family of sets either containing a given element or of size at most $2$.  This extends a result of Keevash and Mubayi, who showed the conclusion for $n$ sufficiently large.
\end{abstract}

\section{Introduction}

Throughout this paper $X$ will be an $n$-element set. For an integer $i\ge 0$, let $X^{(i)} = \{A \subseteq X : |A| = i\}$, and $X^{(\le i)} = \cup_{0 \le j \le i} X^{(j)}$.  If $\ca{F} \subseteq X^{(\le n)}$ and $x \in X$, we let $\ca{F}_x = \{A \in \ca{F} : x \in A\}$ and $\ca{F}-x = \ca{F}\setminus \ca{F}_x$.

A $d$-dimensional simplex, or $d$-simplex, is a collection of $d+1$ sets $A_1,\ldots,A_{d+1}$ such that $\cap_{i=1}^{d+1} A_i = \emptyset$ but $\cap_{i \ne j} A_i \ne \emptyset$ for $1 \le j \le d+1$.  For positive integers $n,d,r$, let
\begin{eqnarray*}
f(n,d) &=& \max\{ \ca{F} \subseteq X^{(\le n)} : \ca{F} \text{ is } d\text{-simplex-free}\}, \text{ and }\\
f_r(n,d) &=& \max\{ \ca{F} \subseteq X^{(r)} : \ca{F} \text{ is } d\text{-simplex-free}\}.
\end{eqnarray*}

The problem of determining $f(n,d)$ and $f_r(n,d)$ can be traced to some of the most fundamental results in extremal combinatorics.  As a $1$-simplex is a pair of nonempty disjoint sets, it is easy to see that $f(n,1) = 2^{n-1} + 1$, while the solution to determining $f_r(n,1)$ comes from the celebrated Erd\H os-Ko-Rado Theorem:

\begin{theorem}[Erd\H os-Ko-Rado \cite{EKR}]
Let $n \ge 2r$ and suppose $\ca{F} \subseteq X^{(r)}$ is intersecting: then $|\ca{F}| \le \bin{n-1}{r-1}$.  If $n > 2r$ and equality holds, then $\ca{F} = X^{(r)}_x$ for some $x \in X$.
\end{theorem}

For $r=d=2$, the forbidden family is a triangle (in graphs), and thus $f_2(n,2) = \lfloor n^2/4\rfloor$, a special case of Tur\'an's theorem and a cornerstone of extremal graph theory.   Erd\H os later posed the question of determining the size of the largest $r$-uniform hypergraph without a triangle ($2$-simplex), i.e. $f_r(n,2)$.  Chv\'atal \cite{C} solved the case $r=3$ by showing the stronger result that for $n \ge r + 2 \ge 5$, $f_r(n,r-1) = \bin{n-1}{r-1}$ with equality only for $\ca{F} = X_x^{(r)}$ for some $x \in X$.  Chv\'atal further conjectured the following:

\begin{conjecture}[Chv\'atal \cite{C}]\label{conj:chvatal}
Let $r \ge d+1 \ge 3$, $n \ge r(d+1)/d$, and $\ca{F} \subseteq X^{(r)}$ with no $d$-simplex.  Then $|\ca{F}| \le \bin{n-1}{r-1}$, with equality only if $\ca{F} = X^{(r)}_x$ for some $x \in X$.
\end{conjecture}

Progress was made on the $d=2$ case over a number of years before it was finally settled by Mubayi and Verstra\"ete (see \cite{MV} for the result and further references).  For $d \ge 3$, Frankl and F\"uredi \cite{FF} established Conjecture \ref{conj:chvatal} for $n$ sufficiently large (see also \cite{JPY}), and, more recently, Keevash and Mubayi \cite{KM} confirmed it when $r$ and $n/2 - r$ are bounded away from $0$.

Erd\H os also posed the question of forbidding triangles in nonuniform systems, which was answered by Milner (unpublished), who showed $f(n,2) = 2^{n-1} + n$ for all $n \ge 1$.  Short proofs of the bound were also found by Lossers \cite{ELM} and by Mubayi and Verstra\"ete \cite{MV} - the latter result establishing that the unique extremal family consists of all sets either containing a given element or of size at most $1$.

For $d \ge 3$, Keevash and Mubayi completely determined $f(n,d)$ and the extremal family for $n$ sufficiently large:

\begin{theorem}[Keevash and Mubayi \cite{KM}]\label{thm:km}
Let $d \ge 2$, and suppose $\ca{F} \subseteq X^{(\le n)}$ is $d$-simplex-free, where $n$ is sufficiently large.  Then $|\ca{F}| \le 2^{n-1} + \sum_{i=0}^{d-1} \bin{n-1}{i}$, with equality if and only if $\ca{F} = X^{(\le n)}_x \cup (X \setminus \{x\})^{(\le d-1)}$ for some $x \in X$.
\end{theorem}

\noindent The proof of Theorem \ref{thm:km} for $d \ge 3$ relies on a stability result that in turn relies on their solution to the uniform problem mentioned above.  Our contribution is to completely determine $f(n,3)$ and the associated extremal family using a simpler inductive argument.

\begin{theorem}\label{thm:mainresult}
For $n \ge 1$, suppose $\ca{F} \subseteq X^{(\le n)}$ is $3$-simplex-free.  Then $|\ca{F}| \le 2^{n-1} +\sum_{i=0}^2 \bin{n-1}{i}$, with equality if and only if $\ca{F} = X^{(\le n)}_x \cup (X\setminus \{x\})^{(\le 2)}$ for some $x \in X$.
\end{theorem}

\section{The Proof of Theorem \ref{thm:mainresult}}

For $d \ge 1$ and $n,k \ge 0$, let $f(n,d,k)$ be the maximum size of a $d$-simplex-free family $\ca{F} \subseteq X^{(\le n-k)}$, and let $g(n,d,k)$ be the maximum size of such a family $\ca{F}$ with $\ca{F} \cap X^{(n-k)} \ne \emptyset$.
Thus $f(n,d,k) = \max_{j \ge k} g(n,d,j)$, and as $X$ cannot lie in a $d$-simplex in $\ca{F}$, $f(n,d) = 1 + f(n,d,1)$.
We begin our arguments with a simple lemma.

\begin{lemma} \label{lem:ineq}
Let $n \ge k \ge 1$ and $d \ge 2$.  Then
\begin{equation}\label{eq:mainineq}
g(n,d,k) \le f(n-k,d)+ \sum_{i=1}^k \bin{k}{i} f(n-k,d-1,i).
\end{equation}
\end{lemma}

\begin{proof}
Let $\ca{F} \subseteq X^{(\le n-k)}$ be $d$-simplex-free with $|\ca{F}| = g(n,d,k)$ and $\max_{A \in \ca{F}} |A| = n-k$, and fix a $Y \in \ca{F}$ with $|Y| = n-k$.  Let $Z = X \setminus Y$, and for every $W \subseteq Z$, let $\ca{F}_W = \{A \cap Y : A \in \ca{F}, A \cap Z = W\}$, so $|\ca{F}| = \sum_{W \subseteq Z} |\ca{F}_W|$.

Clearly $\ca{F}_{\emptyset}$ must be $d$-simplex-free, so $|\ca{F}_{\emptyset}|\le f(n-k,d)$.  Now, fix any nonempty $W \subseteq Z$. If $\ca{F}_W$ contains a $(d-1)$-simplex $A_1,\ldots,A_d$, then letting $B_i = A_i \cup W$ for $1 \le i \le d$ and $B_{d+1} = Y$, the $B_i$ form a $d$-simplex in $\ca{F}$, a contradiction.  By the choice of $Y$, it follows that every $A \in \ca{F}_W$ has size at most $n-k-|W|$, and hence $|\ca{F}_W| \le f(n-k,d-1,|W|)$ and the result follows.\\
\end{proof}

We next show the following simple result for the $d=1$ case, whose proof we include for completeness.

\begin{claim}\label{clm:d=1}
For every $n \ge k \ge 1$, $g(n,1,k)=f(n,1,k) = 2^{n-1} - \sum_{j=1}^{k-1} \bin{n-1}{j}$.
\end{claim}

\begin{proof}
Let $\ca{F} \subseteq X^{(\le n-k)}$ be $1$-simplex-free with $\max_{A \in \ca{F}} |A| = n - k$ and $|\ca{F}| = g(n,1,k)$.
Let $\ca{P}$ be a partition of $\{1,2,\ldots,n-k\}$ into singletons $\{i\}$ with $i \le \frac{n}{2}$ and pairs $\{i,n-i\}$ with $i < n-i$.  Finally, let $\ca{F}^i = \ca{F} \cap X^{(i)}$.

For every singleton $\{i\} \in \ca{P}$, by Erd\H os-Ko-Rado, $|\ca{F}^i| \le \bin{n-1}{i-1} = \bin{n-1}{n-i}$.  For every pair $\{i,n-i\} \in \ca{P}$, as $\ca{F}\setminus \{\emptyset\}$ is intersecting, $|\ca{F}^i| + |\ca{F}^{n-i}| \le \bin{n}{i} = \bin{n-1}{i-1} + \bin{n-1}{i} = \bin{n-1}{n-i} + \bin{n-1}{n-(n-i)}$. As $|\ca{F}^0| \le 1$, it follows that \[|\ca{F}| \le \sum_{i=0}^{n-k} |\ca{F}^i| \le 1 + \sum_{i=1}^{n-k} \bin{n-1}{n-i} = \bin{n-1}{0} + \sum_{i=k}^{n-1} \bin{n-1}{i} = 2^{n-1} - \sum_{i=1}^{k-1} \bin{n-1}{i}.\]

To see that equality holds in the bound, let $\ca{F} = X^{(\le n-k)}_x \cup \{\emptyset\}$ for any $x \in X$.  We also note that for $k \ge 2$, equality implies $k=n$ or $|\ca{F}^1| = 1$ and hence this is the unique extremal family.\\
\end{proof}

Next, we prove a slight strengthening of Milner's result on triangle-free set systems.

\begin{lemma}\label{lem:milner}
For all $n \ge 1$,
\begin{eqnarray}
f(n,2,1) &=& \label{eq:fn21}2^{n-1} + \bin{n-1}{1},\\
f(n,2,2) &\le& \label{eq:fn22} 2^{n-1} + 1, \text{ and}\\
f(n,2,3) &\le& \label{eq:fn23} 2^{n-1}.
\end{eqnarray}
In particular, $f(n,2) = 2^{n-1} + \bin{n-1}{0} + \bin{n-1}{1}$.  Moreover, if $\ca{F}\subseteq X^{(\le n)}$ is triangle-free and $|\ca{F}| = f(n,2)$, then $\ca{F} = X_x^{(\le n)} \cup (X\setminus \{x\})^{(\le 1)}$ for some $x \in X$.
\end{lemma}

\begin{proof}
Our proof is by induction on $n$; it is easy to verify for $n \le 3$, so suppose $n \ge 4$.  As $g(n,2,n) = 1$, let $1 \le k \le n-1$.  Applying Lemma \ref{lem:ineq} and Claim \ref{clm:d=1},
\begin{eqnarray}
g(n,2,k) &\le&\nonumber f(n-k,2) + \sum_{i=1}^{k} \bin{k}{i}f(n-k,1,i)\\
&\le&  \nonumber f(n-k,2) + \sum_{i=1}^k\bin{k}{i} \lp 2^{n-k-1} - \sum_{j=1}^{i-1} \bin{n-k-1}{j} \rp\\
&\le& \nonumber f(n-k,2) + (2^k - 1)2^{n-k-1} - \bin{k}{2}\bin{n-k-1}{1}\\
&=&  \label{eq:ineqd2} 2^{n-1} + \bin{n-k-1}{0} + \bin{n-k-1}{1} \lp 1 - \bin{k}{2} \rp.
\end{eqnarray}

From \eqref{eq:ineqd2}, it follows that $g(n,2,k) \le 2^{n-1} + 1$ for $2 \le k \le n-1$, with equality only possible at $k=2$ or $k=n-1$: as $n \ge 4$, $n-1 > n/2$, so $g(n,2,n-1) \le 2^{n-1}$ and \eqref{eq:fn22} and \eqref{eq:fn23} follow.  The upper bound in \eqref{eq:fn21} also follows from \eqref{eq:ineqd2}, and the lower bound from the conjectured extremal family.

Suppose now that $\ca{F}\subseteq X^{(\le n)}$ is a triangle-free family with $|\ca{F}| = 2^{n-1} + \bin{n-1}{0} + \bin{n-1}{1} > 1 + f(n,2,2)$: it follows that there is a $Y \in \ca{F}$ with $|Y| =n-1$.  Let $z$ be the unique element in $X\setminus Y$.
Define $\ca{F}^1 = \ca{F} - z$ and $\ca{F}^2 = \{A \cap Y :  A \in \ca{F}_z\}$.  As equality holds in \eqref{eq:fn21} (with $k=1$), it follows from the proof of Lemma \ref{lem:ineq} that $\ca{F}^1$ is triangle-free, $\ca{F}^2$ is $1$-simplex-free, $|\ca{F}^1| = f(n-1,2)$ and $|\ca{F}^2| = 1 + f(n-1,1,1) = f(n-1,1)$.  By the induction hypothesis,  $\ca{F}^1 = Y^{(\le n-1)}_y \cup (Y \setminus \{y\})^{(\le 1)}$ for some $y \in Y$.

Suppose there is an $A \in \ca{F}$ with $|A| \ge 2$ and $y \notin A$: then $z \in A$, so $A = \{z,w_1,w_2,\ldots,w_s\}$ for some $s \ge 1$.  If $s \ge 2$, then the sets $\{y,w_1\}, \{y,w_2\}$ and $A$ form a triangle in $\ca{F}$, a contradiction.  If $s = 1$, then $\{w_1\} \in \ca{F}^2$, implying that $\ca{F}^2 = Y^{(\le n-1)}_{w_1} \cup \{\emptyset\}$. As $|Y| \ge 3$, let $w_2 \in Y \setminus \{w_1,y\}$: then $\{z,w_1,y\},\{z,w_1,w_2\},\{w_2,y\}$ lie in $\ca{F}$ and form a triangle, a contradiction.  Therefore $\ca{F} \subseteq  X_y^{(\le n)} \cup (X\setminus \{y\})^{(\le 1)}$ and hence equality holds.\\
\end{proof}

Now that the pieces are in place, we prove the main result.

\begin{proof}[Proof of Theorem \ref{thm:mainresult}]

We use the same inductive approach as in the proof of Lemma \ref{lem:milner}.  We note that the result is trivial for $n< 4$, and for $n=4$ the only restriction is that a single $3$-element set must be missing.  Therefore, assume $n \ge 5$: let $\ca{F} \subseteq X^{(\le n)}$ be $3$-simplex-free with $|\ca{F}| \ge 2^{n-1} + \bin{n-1}{0} + \bin{n-1}{1} + \bin{n-1}{2}$.  Let $Y \in \ca{F}\setminus \{X\}$ have maximum size, and let $k =n- |Y|$.
Then by Lemmas \ref{lem:ineq} and \ref{lem:milner},
\begin{eqnarray*}
|\ca{F}|-1 &\le & g(n,3,k)\\
&\le& f(n-k,3) + \sum_{i=1}^k \bin{k}{i} f(n-k,2,i)\\
&\le& f(n-k,3) + (2^k - 1)2^{n-k-1} + \bin{k}{1} \bin{n-k-1}{1} + \bin{k}{2}\\
&=& f(n-k,3) + (2^k - 1)2^{n-k-1} + \bin{n-1}{2} - \bin{n-k-1}{2}\\
&=& 2^{n-1} + \bin{n-k-1}{0} + \bin{n-k-1}{1} + \bin{n-1}{2}\\
&=& 2^{n-1} + \bin{n-k}{1} + \bin{n-1}{2},
\end{eqnarray*} which by our lower bound on $|\ca{F}|$ implies equality holds throughout and $k=1$.

Let $z$ be the unique element in $X\setminus Y$ and let $\ca{F}^1 = \ca{F} - z$ and $\ca{F}^2 = \{A \setminus \{z\}: A \in \ca{F}_z\}$: then $\ca{F}^1$ is $3$-simplex-free and of size $f(n-1,3)$, and $\ca{F}^2$ is triangle-free and of size $f(n-1,2,1) + 1= f(n-1,2)$.  By Lemma \ref{lem:milner} and the induction hypothesis, there exist $y_1,y_2 \in Y$ such that $\ca{F}^1 = Y_{y_1}^{(\le n-1)} \cup (Y\setminus \{y_1\})^{(\le 2)}$ and $\ca{F}^2 = Y_{y_2}^{(\le n-1)} \cup (Y\setminus \{y_2\})^{(\le 1)}$.
As every set in $\ca{F}^2$ of size at most $1$ corresponds to a set of size at most $2$ in $\ca{F}$, it suffices to show that $y_1=y_2$, so suppose otherwise.  As $n \ge 5$, $|Y| \ge 4$, so let $w_1,w_2 \in Y \setminus \{y_1,y_2\}$: then the sets $\{y_1,y_2,w_1\}, \{y_1,y_2,w_2\}, \{y_1,w_1,w_2\}$ and $\{z,y_2,w_1,w_2\}$ all lie in $\ca{F}$ and form a $3$-simplex, a contradiction.\\
\end{proof}

\section{Concluding Remarks}

Complications arise in attempting to extend this method to forbidding $d$-simplices with $d \ge 4$, the chief among them following from the fact that for $k \ge 2$, \eqref{eq:mainineq} is not, in general, sharp.  To see this and illustrate the difficulty with the $d=4$ case, note that by the extremal family, $f(n,3,2) \ge 2^{n-1} + \bin{n-1}{2}$ for $n \ge 4$.  With similar calculations as above, this implies the best upper bound on $g(n,4,2)$ guaranteed by \eqref{eq:mainineq} is $2^{n-1} + \bin{n-1}{2} + \bin{n-1}{3} + \bin{n-3}{2}$, which is greater than $2^{n-1} + \bin{n-1}{1} + \bin{n-1}{2} + \bin{n-1}{3}$ for all $n \ge 8$.

However, we suspect that, in general and for most $k$, the family $X_x^{(\le n-k)} \cup (X \setminus \{x\})^{(\le d-1)}$ determines $f(n,d,k)$ and $g(n,d,k)$:

\begin{conjecture}\label{conj:main}
For $d \ge 2$ and $1 \le k \le n-d-1$, if $\ca{F} \subseteq X^{(\le n-k)}$ is $d$-simplex-free, then
\[|\ca{F}| \le 2^{n-1} + \sum_{i=0}^{d-1} \bin{n-1}{i} - \sum_{i=0}^{k-1} \bin{n-1}{i},\]
with equality if and only if $\ca{F} = X_x^{(\le n-k)} \cup (X \setminus \{x\})^{(\le d-1)}$ for some $x \in X$.
\end{conjecture}

\noindent We mention that the proof of Theorem \ref{thm:km} yields that Conjecture \ref{conj:main} holds for $k \le d$ provided $n$ is sufficiently large.

\end{document}